\documentclass[11pt,a4paper,reqno]{amsart}
\usepackage{hyperref}
\usepackage[all,cmtip]{xy}
\usepackage{amssymb}
\usepackage{amsmath}

\usepackage{algorithm}
\usepackage[noend]{algpseudocode}

\newcommand{\fa}{\mathfrak{a}}

\newcommand{\fm}{\mathfrak{m}}

\newcommand{\af}{\text{af}}

\chardef\bslash=`\\ 

\newtheorem{theorem}{Theorem}[section]

\newtheorem{prop}[theorem]{Proposition}
\newtheorem{lem}[theorem]{Lemma}
\newtheorem{cor}[theorem]{Corollary}

\theoremstyle{definition}

\newtheorem{remark}[theorem]{Remark}
\newtheorem{example}[theorem]{Example}

\numberwithin{equation}{section}

\newtheorem*{maintheorem*}{Main Theorem}
\theoremstyle{definition}
\newtheorem{definition}{Definition}

\newcommand{\CO}{\mathcal{O}}
\newcommand{\OV}{\un{\textbf{O}}_V}
\newcommand{\SOV}{\un{\textbf{SO}}_V}

\newcommand{\ClS}{\text{Cl}_S}

\newcommand{\CS}{\CO_S}

\renewcommand{\sectionmark}[1]{}

\newcommand{\Br}{{\mathrm{Br}}}

\newcommand{\iy}{\infty}

\newcommand{\Pic}{\text{Pic~}}

\newcommand{\dl}{\delta}

\newcommand{\un}{\underline}

\newcommand{\BF}{\mathbb{F}}

\newcommand{\et}{\text{\'et}}

\begin{document}

\date{}


\baselineskip20pt
\setcounter{equation}{0}
\pagestyle{plain}
\pagenumbering{arabic}

\title{On the classification of quadratic forms over an integral domain of a global function field}

\author{Rony A. Bitan}
\thanks{This work was supported by a Chateaubriand Fellowship of the Embassy of France in Israel, 2016.}

\begin{abstract}
Let $C$ be a smooth projective curve defined over the finite field $\BF_q$ ($q$ is odd)
and let $K=\BF_q(C)$ be its function field. 
Any finite set $S$ of closed points of $C$ gives rise to an integral domain $\CS:=\BF_q[C-S]$ in $K$.  
We show that given an $\CS$-regular quadratic space $(V,q)$ of rank $n \geq 3$, 
the set of genera in the proper classification of quadratic $\CS$-spaces  
isomorphic to $(V,q)$ in the flat or \'etale topology, 
is in $1:1$ correspondence with ${_2}\Br(\CS)$, thus there are $2^{|S|-1}$ such.   
If $(V,q)$ is isotropic, then $\Pic(\CS)/2$ classifies the forms in the genus of $(V,q)$.   
For $n \geq 5$ this is true for all genera, hence the full classification is via the abelian group $H^2_\et(\CS,\un{\mu}_2)$.  
\end{abstract}

\maketitle
\markright{Classification of integral forms} 

\pagestyle{headings}

\section{Introduction}

Let $C$ be a projective algebraic curve defined over a finite field
$\mathbb{F}_{q}$ (with $q$ odd), assumed to be geometrically connected
and smooth. Let $K=\mathbb{F}_{q}(C)$ be its function field and let
$\Omega$ denote the set of all closed points of $C$. For any point
$\mathfrak{p}\in\Omega$ let $v_{\mathfrak{p}}$ be the induced discrete
valuation on $K$, $\hat{\mathcal{O}}_{\mathfrak{p}}$ the complete
discrete valuation ring with respect to $v_{\mathfrak{p}}$ and
$\hat{K}_{\mathfrak{p}}$ its fraction field. Any \emph{Hasse set} of
$K$, namely a non-empty finite set $S \subset\Omega$, gives rise to
an integral domain of $K$ called a \emph{Hasse domain}:
\begin{equation*}
\mathcal{O}_{S}:= \{x \in K: v_{\mathfrak{p}}(x) \geq0 \ \forall
\mathfrak{p}\notin S\}.
\end{equation*}
This is a Dedekind domain, regular and one dimensional. Schemes defined
over $\text{Spec} \,\mathcal{O}_{S}$ are denoted by an underline, being
omitted in the notation of their generic fibers.

As $2$ is invertible in $\mathcal{O}_{S}$, the $\mathcal{O}_{S}$-group
$\underline{\mu}_{2} := \text{Spec} \,\mathcal{O}_{S}[t]/(t^{2}-1)$ is
smooth, whence applying \'{e}tale cohomology to the Kummer sequence:
\begin{equation*}
1 \to\underline{\mu}_{2} \to\underline{\mathbb{G}}_{m}
\xrightarrow{x \mapsto x^{2}} \underline{\mathbb{G}}_{m} \to1
\end{equation*}
gives rise to the long exact sequence of abelian groups:
\begin{eqnarray*}
H^{1}_{\text{\'{e}t}}(\mathcal{O}_{S},\underline{\mathbb{G}}_{m})
&\xrightarrow{[ \mathfrak{L}] \to[\mathfrak{L} \otimes_{\mathcal{O}_{S}}
\mathfrak{L}]}&
H^{1}_{\text{\'{e}t}}(\mathcal{O}_{S},\underline{\mathbb{G}}_{m}) \to
H^{2}_{\text{\'{e}t}}(\mathcal{O}_{S},\underline{ \mu}_{2}) \to
H^{2}_{\text{\'{e}t}}(\mathcal{O}_{S},\underline{ \mathbb{G}}_{m})\\
&\xrightarrow{[A] \to[A \otimes_{\mathcal{O}_{S}} A]}& H^{2}_{
\text{\'{e}t}}(\mathcal{O}_{S},\underline{\mathbb{G}}_{m}).
\end{eqnarray*}
Identifying
$H^{1}_{\text{\'{e}t}}(\mathcal{O}_{S},\underline{\mathbb{G}}_{m})$
with $\text{Pic~}(\mathcal{O}_{S})$ by Shapiro's lemma (cf.
\cite[XXIV,\break Prop.~8.4]{SGA3}), and $H^{2}_{\text{\'{e}t}}(
\mathcal{O}_{S},\underline{\mathbb{G}}_{m})$ with the Brauer group
${\mathrm{Br}}(\mathcal{O}_{S})$, classifying Azumaya $\mathcal{O}
_{S}$-algebras (cf. \cite[\S2]{Mil}), we deduce the short exact
sequence:
%
\begin{equation}
\label{Kummermu2H1}
1 \to\text{Pic~}(\mathcal{O}_{S})/2 \xrightarrow{\partial} H^{2}
_{\text{\'{e}t}}(\mathcal{O}_{S},\underline{\mu}_{2}) \xrightarrow{i
_{*}^{2}} {_{2}}{\mathrm{Br}}(\mathcal{O}_{S}) \to1,
\end{equation}
in which the right non-trivial term is the $2$-torsion part in
${\mathrm{Br}}(\mathcal{O}_{S})$. We analyze some properties related to
this sequence in Section~\ref{Azumayasalgebras}, which will be used to
classify regular quadratic $\mathcal{O}_{S}$-spaces.

Let $(V,q)$ (not to be confused with $q=|\mathbb{F}_{q}|$) be a
quadratic $\mathcal{O}_{S}$-space of rank $n \geq3$, namely, $V$ is a
projective $\mathcal{O}_{S}$-module of rank $n$ and $q:V \to
\mathcal{O}_{S}$ is a $2$-order homogeneous $\mathcal{O}_{S}$-form.
Since $2$ is a unit, $q$ corresponds to the symmetric bilinear form
$B_{q}:V \times V \to\mathcal{O}_{S}$ such that:
\begin{equation*}
B_{q}(u,v) = q(u+v)-q(u)-q(v) \ \ \forall u,v \in V.
\end{equation*}
We assume it to be $\mathcal{O}_{S}$-\emph{regular}, namely, the induced
homomorphism $V \to V^{\vee}:= \text{Hom}(V,\mathcal{O}_{S})$ is an
isomorphism (\cite[I.~\S3, 3.2]{Knu}). 
Two quadratic $\mathcal{O}_{S}$-spaces $(V,q)$ and $(V',q')$ are \emph{isomorphic} over an
extension $R$ of $\mathcal{O}_{S}$, if there exists an
$R$-\emph{isometry} between them, namely, an $R$-isomorphism
$T: V' \otimes_{\mathcal{O}_{S}} R \cong V \otimes_{\mathcal{O}_{S}} R$
such that $q \circ T = q'$. The notation of the
$\mathcal{O}_{S}$-isomorphism class $[(V,q)]$ is sometimes, when no
ambiguity arises, shortened to $[q]$. 
The [proper] \emph{genus} of $(V,q)$ is the set of classes of all quadratic $\mathcal{O}_{S}$-spaces
that are [properly, i.e.,\vadjust{\goodbreak} with $\det=1$ isomorphisms] isomorphic to
$(V,q)$ over $K$ and over $\hat{\mathcal{O}}_{\mathfrak{p}}$ for any
prime $\mathfrak{p} \notin S$. 
This [proper] genus bijects as a pointed-set with
the \emph{class set} $[\text{Cl}_{S}(\underline{\text{\textbf{SO}}}_{V})]$,
$\text{Cl}_{S}(\underline{\text{\textbf{O}}}_{V})$.

The results generalize the ones in \cite{Bit}, in which $S$ is assumed
to contain only one (arbitrary) point $\infty\in\Omega$, thus giving
rise to an affine curve $C^{\text{af}}= C-\{\infty\}$, for which
${\mathrm{Br}}(\mathcal{O}_{\{\infty\}})=1$ (cf.
\cite[Lemma~3.3]{Bit}). The quadratic $\mathcal{O}_{\{\infty\}}$-spaces
that are locally properly isomorphic to $(V,q)$ for the flat or the
\'{e}tale topology belong all to the same genus, and are classified by
the abelian group $H^{2}_{\text{\'{e}t}}(\mathcal{O}_{\{\infty\}
},\underline{\mu}_{2}) \cong\text{Pic~}(\mathcal{O}_{\{\infty\}})/2$.
Here we show more generally for any finite set $S$ that, as
$\text{Cl}_{S}(\underline{\text{\textbf{SO}}}_{V})$ is the kernel of
what we call the \emph{relative} Witt-invariant
$H^{1}_{\text{\'{e}t}}(\mathcal{O}_{S},\underline{\text{\textbf{SO}}}_{V})
\xrightarrow{w_{V}} {_{2} {\mathrm{Br}}( \mathcal{O}_{S})}$, the latter
abelian group bijects to the set of $2^{|S|-1}$ proper genera of
$(V,q)$ (Proposition~\ref{Wittexactness}).

Another consequence of passing to $|S|>1$ is that $\mathcal{O}_{S}
^{\times}\neq\mathbb{F}_{q}^{\times}$, whence
$\mathcal{O}_{S}$-regularity imposed on $(V,q)$, no longer guarantees
its isotropy. Requiring $(V,q)$ to be isotropic, we show that
$\text{Pic~}(\mathcal{O}_{S})/2$ still classifies the quadratic spaces
in the genus $\text{Cl}_{S}(\underline{\text{\textbf{O}}}_{V})$ being
equal to proper genus in this case (Lemma~\ref{samegenus}), for any $S$
(Theorem~\ref{isotropicclassification}). In particular in case $n \geq 5$, 
in which all classes are isotropic, any proper genus in
$H^{1}_{\text{\'{e}t}}(\mathcal{O}_{S},\underline{\text{\textbf{SO}}}_{V})$ --
corresponding as aforementioned to an element of $_{2} {\mathrm{Br}}(
\mathcal{O}_{S})$ -- is isomorphic to $\text{Pic~}(\mathcal{O}_{S})/2$,
whence their disjoint union $H^{1}_{\text{\'{e}t}}(\mathcal
{O}_{S},\underline{\text{\textbf{SO}}}_{V})$ is isomorphic to the
abelian group $H^{2}_{\text{\'{e}t}}(\mathcal{O}_{S},\underline{\mu}_{2})$ (as for $|S|=1$
and $n \geq3$), fitting into the sequence \eqref{Kummermu2H1}
(Corollary~\ref{genera}).

In Section~\ref{hyperbolicplane}, we refer to the case in which $V$ is
split by a hyperbolic plane $H(L_{0})$, and provide an isomorphism
$\psi_{V}:\text{Pic~}(\mathcal{O}_{S})/2
\to\text{Cl}_{S}(\underline{\text{\textbf{O}}}_{V})$. In case $C$ is an
elliptic curve and $S=\{\infty \}$ where $\infty$ is
$\mathbb{F}_{q}$-rational, an algorithm, producing explicitly
representatives of classes in $H^{1}_{
\text{\'{e}t}}(\mathcal{O}_{S},\underline{\text{\textbf{SO}}}_{V})$, is given (\ref{Generatoralgorithm}).

\section{A classification of Azumaya algebras}\label{Azumayasalgebras}

A faithfully flat projective (right) $\mathcal{O}_{S}$-module $A$ is an
\emph{Azumaya} $\mathcal{O}_{S}$-algebra if the map
\begin{equation*}
A \otimes A^{\text{op}} \to\text{End}_{\mathcal{O}_{S}}(A): \ a
\otimes b^{\text{op}} \mapsto(x \mapsto axb)
\end{equation*}
is an isomorphism. It is central, separable and finitely generated as
an $\mathcal{O}_{S}$-module. Two Azumaya $\mathcal{O}_{S}$-algebras
$A,B$ are \emph{Brauer equivalent} if there exist faithfully projective
modules $P,Q$ such that:
\begin{equation*}
A \otimes\text{End}_{\mathcal{O}_{S}}(P) \cong B \otimes\text{End}
_{\mathcal{O}_{S}}(Q).
\end{equation*}
The tensor product induces the structure of an abelian group
${\mathrm{Br}}(\mathcal{O}_{S})$ on the equivalence classes, in which
the neutral element is $[\mathcal{O}_{S}]$ and the inverse of $[A]$ is
$[A^{\text{op}}]$ (cf. \cite[III.~5.1 and 5.3]{Knu}).

Let $V \cong\mathcal{O}_{S}^{2}$. Consider the following exact and
commutative diagram of smooth $\mathcal{O}_{S}$-groups:
%
\begin{align}
\label{diagram}
\xymatrix{
& 1 \ar[d] & 1 \ar[d] \\
1 \ar[r] & \underline{\mu}_2 \ar[r] \ar[d] & \underline{\text{\textbf{SL}}}(V) \ar[r] \ar[d]
& \underline{\text{\textbf{PGL}}}(V) \ar[r] \ar@{=}[d] & 1 \\
1 \ar[r] & \underline{\mathbb{G}}_m \ar[r] \ar[d]^{x \mapsto x^2} & \underline{\textbf
{GL}}(V) \ar[r]^{\pi} \ar[d]^{\det} & \underline{\text{\textbf{PGL}}}(V) \ar[r] & 1 \\
& \underline{\mathbb{G}}_m \ar[r] \ar[d] & \underline{\mathbb{G}}_m \ar[d] \\
& 1 & 1.
}
\end{align}
The generalization of the Skolem--Noether Theorem to unital commutative
rings, applied to the Azumaya $\mathcal{O}_{S}$-algebra $A=\text{End}
_{\mathcal{O}_{S}}(V)$, is the exact sequence of groups (see
\cite[III.~5.2.1]{Knu}):
\begin{equation*}
1 \to\mathcal{O}_{S}^{\times}\to A^{\times}\to\text{Aut}_{
\mathcal{O}_{S}}(A) \to\text{Pic~}(\mathcal{O}_{S}).
\end{equation*}
This sequence induces by sheafification a short exact sequence of
sheaves in the \'{e}tale topology (cf. \cite[p.~145]{Knu}):
\begin{equation*}
1 \to\underline{\mathbb{G}}_{m} \to\underline{\text{\textbf{GL}}}(V)
\to\underline{\text{Aut}}(\text{End}_{\mathcal{O}_{S}}(V)) \to1
\end{equation*}
from which we see that: $\underline{\text{\textbf{PGL}}}(V) =
\underline{\text{Aut}}(\text{End}_{\mathcal{O}_{S}}(V))$. In this interpretation,
\'{e}tale cohomology applied to the diagram \eqref{diagram}, plus the
sequence \eqref{Kummermu2H1}, give rise to the exact and commutative diagram:
%
\begin{equation}
\label{PGLV(V)}
\xymatrix{
& & \text{Pic~}(\mathcal{O}_S) \ar[d]^\partial \\
H^1_\text{\'et}(\mathcal{O}_S,\underline{\text{\textbf{SL}}}(V)) \ar[r] \ar[d] & H^1_\text{\'et}(\mathcal{O}_S,\underline{\textbf
{PGL}}(V)) \ar[r]^{\partial^1} \ar@{=}[d] & H^2_\text{\'et}(\mathcal{O}_S,\underline{\mu}_2) \ar[d] \\
H^1_\text{\'et}(\mathcal{O}_S,\underline{\text{Aut}}(V)) \ar[r]^{\pi_*} \ar[d]^{\Delta= \det_*} &
H^1_\text{\'et}(\mathcal{O}_S,\underline{\text{Aut}}(\text{End}_{\mathcal{O}_S}(V))) \ar[r] & {\mathrm{Br}}(\mathcal{O}_S) \\
\text{Pic~}(\mathcal{O}_S)
}
\end{equation}
in which $H^{1}_{\text{\'{e}t}}(\mathcal{O}_{S},
\underline{\text{Aut}}(V))$ classifies twisted forms of $V$ in the
\'{e}tale topology, while its image in $H^{1}_{\text{\'{e}t}}(
\mathcal{O}_{S},\underline{\text{Aut}}(\text{End}_{\mathcal{O}_{S}}(V))$
classifies these $\mathcal{O}_{S}$-modules up to scaling by an
$\mathcal{O}_{S}$-line, i.e., by an invertible $\mathcal{O}_{S}$-module.
Explicitly, $\pi_{*}: [P] \mapsto[\text{End}_{\mathcal{O}_{S}}(P)]$
(cf. \cite[III.~5.2.4]{Knu}).

\begin{cor}
\label{PictoH2}
In diagram \eqref{PGLV(V)}: $\partial([L]) = \partial^{1}([
\text{End}_{\mathcal{O}_{S}}(\mathcal{O}_{S}\oplus L)])$.
\end{cor}

\begin{proof}
By chasing diagram \eqref{PGLV(V)} we may deduce the following reduced one:
%
\begin{equation}
\label{reduceddiagram}
\xymatrix{
&& \text{Pic~}(\mathcal{O}_S)/2 \ar@{^{(}->}[d]^\partial \\
H^1_\text{\'et}(\mathcal{O}_S,\underline{\text{Aut}}(V)) \ar[rr]^{\partial^1 \circ\pi_*} \ar[urr]^{\Delta}
\ar[drr]^{0} && H^2_\text{\'et}(\mathcal{O}_S,\underline{\mu}_2) \ar@{->>}[d] \\
&& {_2 {\mathrm{Br}}(\mathcal{O}_S)}
}
\end{equation}
which shows that: $ \partial([L]) = \partial(\Delta([\mathcal{O}
_{S}\oplus L])) = \partial^{1}(\pi_{*}([\mathcal{O}_{S}\oplus L])) =
\partial^{1}([\text{End}_{\mathcal{O}_{S}}(\mathcal{O}_{S}\oplus L)])$.
\end{proof}

\begin{lem}
\label{no.genera}
$|{_{2} {\mathrm{Br}}(\mathcal{O}_{S})}| = 2^{|S|-1}$.
\end{lem}

\begin{proof}
Let $r_{\mathfrak{p}}: {\mathrm{Br}}(K) \to H^{1}(k_{\mathfrak{p}},
\mathbb{Q}/\mathbb{Z}) \cong\mathbb{Q}/\mathbb{Z}$ be the residue map
at a prime $\mathfrak{p}$. The ramification map $a:= \oplus_{
\mathfrak{p}}r_{\mathfrak{p}}$ yields the exact sequence from Class
Field Theory (see \cite[Theorem~6.5.1]{GS}):
%
\begin{equation}
\label{sumHassesinvariants}
1 \to{\mathrm{Br}}(K) \xrightarrow{a} \bigoplus_{
\mathfrak{p}}\mathbb{Q}/ \mathbb{Z}\xrightarrow{\sum_{\mathfrak{p}}
\text{Cor}_{\mathfrak{p}}} \mathbb{Q}/\mathbb{Z}\to1
\end{equation}
in which the corestriction map $\text{Cor}_{\mathfrak{p}}$ for any
$\mathfrak{p}$ is an isomorphism induced by the Hasse-invariant
${\mathrm{Br}}(\hat{K}_{\mathfrak{p}}) \cong\mathbb{Q}/\mathbb{Z}$ (cf.
\cite[Proposition~6.3.9]{GS}). On the other hand, as all residue fields
of $K$ are finite, thus perfect, and $\mathcal{O}_{S}$ is a
one-dimensional regular scheme, it admits due to Grothendieck the
following exact sequence (see \cite[Proposition~2.1]{Gro} and
\cite[Example~2.22, case (a)]{Mil}):
%
\begin{equation}
\label{Grothendieck}
1 \to{\mathrm{Br}}(\mathcal{O}_{S}) \to{\mathrm{Br}}(K) \xrightarrow{
\oplus_{\mathfrak{p}\notin S} r_{\mathfrak{p}}} \bigoplus
_{\mathfrak{p}\notin S} \mathbb{Q}/\mathbb{Z},
\end{equation}
which means that ${\mathrm{Br}}(\mathcal{O}_{S})$ is the subgroup of
${\mathrm{Br}}(K)$ of classes that vanish under $r_{\mathfrak{p}}$ at
any $\mathfrak{p}\notin S$. Thus omitting these $r_{\mathfrak{p}},
\mathfrak{p}\notin S$ in the sequence (\ref{sumHassesinvariants}),
results in ${\mathrm{Br}}(\mathcal{O}_{S}) = \ker\left[ \bigoplus_{
\mathfrak{p}\in S} \mathbb{Q}/\mathbb{Z}\xrightarrow{
\sum_{\mathfrak{p}\in S}\text{Cor}_{\mathfrak{p}}} \mathbb{Q}/
\mathbb{Z}\right] $, whence the cardinality of its $2$-torsion part is
$2^{|S|-1}$.
\end{proof}

\begin{lem}
\label{H1=1sc}
Let $\underline{G}$ be an affine, flat, connected and smooth
$\mathcal{O}_{S}$-group. Suppose that its generic fiber $G$ is almost
simple, simply connected and $\hat{K}_{\mathfrak{p}}$-isotropic for any
$\mathfrak{p}\in S$. Then $H^{1}_{\text{\'{e}t}}(\mathcal{O}_{S},
\underline{G})=0$.
\end{lem}

\begin{proof}
The proof, basically relying on the strong approximation property
related to $\underline{G}$, is similar to that of Lemma~3.2 in \cite{Bit},
replacing $\{\infty\}$ by $S$.
\end{proof}

\section{Standard and relative invariants}\label{forms}

Let $\underline{\text{\textbf{O}}}_{V}$ be the \emph{orthogonal group} of
$(V,q)$ defined over $\text{Spec} \,\mathcal{O}_{S}$, namely, the
functor assigning to any $\mathcal{O}_{S}$-algebra $R$ the group of
self-isometries of $q$ over $R$:
\begin{equation*}
\underline{\text{\textbf{O}}}_{V}(R) = \{ A \in\underline{\text{\textbf{GL}}}_{n}(R):
q \circ A = q \}.
\end{equation*}
Since $2 \in\mathcal{O}_{S}^{\times}$ and $q$ is regular,
$\underline{\text{\textbf{O}}}_{V}$ is smooth as well as its connected
component, namely, the \emph{special orthogonal group}
$\underline{\text{\textbf{SO}}}_{V}:=
\ker[\underline{\text{\textbf{O}}}_{V}\xrightarrow{ \det}
\underline{\mu}_{2}]$ (see Definition~1.6, Theorem~1.7 and
Corollary~2.5 in \cite{Con}). Thus the pointed set $H^{1}_{
\text{fl}}(\mathcal{O}_{S},\underline{\text{\textbf{SO}}}_{V})$ --
properly (i.e., with $\det=1$ isomorphisms) classifying
$\mathcal{O}_{S}$-forms that are locally everywhere isomorphic to $q$
in the flat topology -- coincides with the classification
$H^{1}_{\text{\'{e}t}}(\mathcal{O}
_{S},\underline{\text{\textbf{SO}}}_{V})$ for the \'{e}tale topology
(see \cite[VIII~Corollaire~2.3]{SGA4}).

Let $\text{\textbf{C}}(V) :=T(V)/(v \otimes v -q(v)\cdot1:v \in V)$ be the
\emph{Clifford algebra} associated to $(V,q)$ (see \cite[IV]{Knu}). The
linear map $v \mapsto-v$ on $V$ preserves $q$, thus extends to an
algebra automorphism $\alpha: \text{\textbf{C}}(V) \to\text{\textbf{C}}(V)$. As it
is an involution, the graded algebra $\text{\textbf{C}}(V)$ is decomposed into
positive and negative eigenspaces: $\text{\textbf{C}}_{0}(V) \oplus
\text{\textbf{C}}_{1}(V)$ where $\text{\textbf{C}}_{i}(V) = \{ x \in\text{\textbf{C}}(V):
\alpha(x) = (-1)^{i} x \}$ for $i=0,1$. Since $(V,q)$ is projective and
$\mathcal{O}_{S}$-regular, $\text{\textbf{C}}(V)$ is Azumaya over $
\mathcal{O}_{S}$ (cf. \cite[Theorem, p.~166]{Bas}).

The \emph{Witt-invariant} of $(V,q)$ is:
\begin{equation*}
w(q) = \left\{
\begin{array}{c@{\quad} c}
[\text{\textbf{C}}(V)] \in{\mathrm{Br}}(\mathcal{O}_{S}) & \ n \
\text{is even}
\\
\nonumber
[\text{\textbf{C}}_{0}(V)] \in{\mathrm{Br}}(\mathcal{O}_{S}) & \ n \ \text{is
odd}.
\end{array}
\right.
\end{equation*}
As $\text{\textbf{C}}(V)$ and $\text{\textbf{C}}_{0}(V)$ are algebras with involution,
$w(q)$ lies in ${_{2} {\mathrm{Br}}(\mathcal{O}_{S})}$
(\cite[IV.~8]{Knu}).

The \emph{Clifford group} associated to $(V,q)$ is
\begin{equation*}
\text{\textbf{CL}}(V) := \{ u \in\text{\textbf{C}}(V)^{\times}\ : \ \alpha(u)vu
^{-1} \in V \ \forall v \in V \}.
\end{equation*}
The group $\underline{\text{\textbf{Pin}}}_{V}(\mathcal{O}_{S}):=\ker[
\text{\textbf{CL}}(V) \xrightarrow{N} \mathcal{O}_{S}^{\times}]$ where
$N:v \mapsto v\alpha(v)$, admits an underlying $\mathcal{O}_{S}$-group
scheme, called the \emph{Pinor group} denoted by
$\underline{\text{\textbf{Pin}}}_{V}$. It is a double covering of $
\underline{\text{\textbf{O}}}_{V}$ and its center $\underline{\mu}_{2}$
is smooth. So applying \'{e}tale cohomology to the \emph{Pinor exact
sequence} of smooth $\mathcal{O}_{S}$-groups:
%
\begin{equation}
\label{Pinor}
1 \to\underline{\mu}_{2} \to\underline{\text{\textbf{Pin}}}_{V} \to
\underline{\text{\textbf{O}}}_{V}\to1
\end{equation}
gives rise to the coboundary map of pointed-sets
%
\begin{equation}
\label{H1Pinor} \delta_{V}:
H^{1}_{\text{\'{e}t}}(\mathcal{O}_{S},\underline{\text{\textbf{O}}}_{V})
\to H^{2}_{\text{\'{e}t}}(\mathcal{O}_{S},\underline{ \mu}_{2}).
\end{equation}

Let $\underline{\text{\textbf{O}}}_{2n}$ and $\underline{\text{\textbf{O}}}_{2n+1}$
be the orthogonal groups of the hyperbolic spaces $H(\mathcal{O}_{S}
^{n})$ and $H(\mathcal{O}_{S}^{n}) \bot\langle1 \rangle$,
respectively, equipped with the \emph{standard split form} which we
denote by $q_{n}$ (see \cite[Definition~1.1]{Con}). The pointed set
$H^{1}_{\text{\'{e}t}}(\mathcal{O}_{S},\underline{\text{\textbf{O}}}_{n})$
classifies regular quadratic $\mathcal{O}_{S}$-modules of rank $n$
(\cite[IV.~5.3.1]{Knu}). It is identified with the pointed set
$H^{1}_{\text{\'{e}t}}(\mathcal{O}_{S},\underline{\text{\textbf{O}}}_{V})$
simply obtained by changing the base point to $(V,q)$ (cf.
\cite[IV, Prop.~8.2]{Knu}). We denote this identification by
$\theta$. One has the following commutative diagram of pointed sets
(cf. \cite[IV, Prop.~4.3.4]{Gir}):
%
\begin{equation}
\label{relativew}
\xymatrix{
H^1_\text{\'et}(\mathcal{O}_S,\underline{\text{\textbf{O}}}_{n}) \ar[r]^{\theta}_{\cong} \ar[d]^{\delta} &
H^1_\text{\'et}(\mathcal{O}_S,\underline{\text{\textbf{O}}}_{V}) \ar[d]_{\delta_V} \\
H^2_\text{\'et}(\mathcal{O}_S,\underline{\mu}_2) \ar[r]^{r_V}_{\cong} & H^2_\text{\'et}(\mathcal{O}_S,\underline{\mu}_2)
}
\end{equation}
in which $\delta:= \delta_{q_{n}}$ and $r_{V}(x)=x-\delta([q])$.

\begin{definition}
\label{relativeWitt-invariant}
We call the composition of maps of pointed sets:
\begin{equation*}
w_{V}: H^{1}_{\text{\'{e}t}}(\mathcal{O}_{S},\underline{\text{\textbf{O}}}
_{V}) \xrightarrow{\delta_{V}} H^{2}_{\text{\'{e}t}}(\mathcal{O}
_{S},\underline{\mu}_{2}) \xrightarrow{i_{*}^{2}} {_{2}}{\mathrm{Br}}(
\mathcal{O}_{S})
\end{equation*}
(see sequence \eqref{Kummermu2H1} for $i_{*}^{2}$) the \emph{relative
Witt-invariant}. It is a ``shift'' of the Witt-invariant $w=i_{*}^{2}
\circ\delta$, such that the base-point $[(V,q)]$ is mapped to
$[0] \in{\mathrm{Br}}(\mathcal{O}_{S})$.
\end{definition}

\begin{remark}
\label{Stiefel-Whitney}
The $\delta$-image of a class represented by $(V',q')$, being a regular
$\mathcal{O}_{S}$-module, is its \emph{second Stiefel--Whitney class},
denoted $w_{2}(q')$ (cf.
\cite[Definition~1.6 and Corollary~1.19]{EKV}).
\end{remark}

The connected component $\underline{\text{\textbf{Spin}}}_{V}$ of
$\underline{\text{\textbf{Pin}}}_{V}$ is smooth, and it is the
universal covering of $\underline{\text{\textbf{SO}}}_{V}$:
%
\begin{equation}
\label{Spinshortsequence}
1 \to\underline{\mu}_{2} \to\underline{\text{\textbf{Spin}}}_{V} \stackrel{
\pi}{\rightarrow} \underline{\text{\textbf{SO}}}_{V}\to1.
\end{equation}
Then \'{e}tale cohomology gives rise to the exact sequence of pointed
sets:
%
\begin{equation}
\label{H1Spin}
H^{1}_{\text{\'{e}t}}(\mathcal{O}_{S},\underline{\text{\textbf{Spin}}}_{V})
\to
H^{1}_{\text{\'{e}t}}(\mathcal{O}_{S},\underline{\text{\textbf{SO}}}_{V})
\xrightarrow{s\delta_{V}} H^{2}_{\text{\'{e}t}}(\mathcal
{O}_{S},\underline{\mu}_{2}) \to1
\end{equation}
in which the right exactness comes from the fact that $\mathcal{O}
_{S}$ is of Douai-type, thus $H^{2}_{\text{\'{e}t}}(\mathcal
{O}_{S},\underline{\text{\textbf{Spin}}}_{V})=1$ (see Definition~5.2
and Example~5.4(iii) in \cite{Gon}). The inclusion
$i:\underline{\text{\textbf{SO}}}_{V}\subset
\underline{\text{\textbf{O}}}_{V}$ with the map $i_{*}^{2}$ from
sequence (\ref{Kummermu2H1}) induces the commutative diagram (cf.
\cite[IV, 8.3]{Knu})
%
\begin{equation}
\label{sdl}
\xymatrix{
H^1_\text{\'et}(\mathcal{O}_S,\underline{\text{\textbf{SO}}}_{V}) \ar[r]^{i_*} \ar[d]^{s\delta_V} & H^1_\text{\'et}(\mathcal{O}_S,\underline{\text{\textbf{O}}}_{V}) \ar
[d]_{w_V} \ar[ld]^{\delta_V} \\
H^2_\text{\'et}(\mathcal{O}_S,\underline{\mu}_2) \ar[r]_{i_*^2} & {_2}{\mathrm{Br}}(\mathcal{O}_S).
}
\end{equation}

\begin{remark}
\label{restriction} The map $i_{*}$ does not have to be injective, yet
any form $q'$, properly isomorphic to $q$, represents a class in
$H^{1}_{
\text{\'{e}t}}(\mathcal{O}_{S},\underline{\text{\textbf{O}}}_{n})$, so
the restriction of $w$ to
$H^{1}_{\text{\'{e}t}}(\mathcal{O}_{S},\underline{\text{\textbf{SO}}}_{n})$
is well-defined. Similarly, we may write the restriction
$w_{V}|H^{1}_{\text{\'{e}t}}(\mathcal{O}_{S},\underline{\text{\textbf{SO}}}_{V})$
as $i_{*}^{2} \circ s \delta_{v}$, being surjective, as both
$i_{*}^{2}$ and $s \delta_{V}$ are such (see sequences \eqref{Kummermu2H1} and \eqref{H1Spin}).
\end{remark}

\begin{remark}
\label{projectiveplane1}
Unlike over fields, the Stiefel--Whitney class $w_{2}$ for quadratic
$\mathcal{O}_{S}$-spaces, referring to their Clifford algebras not only
as Azumaya algebras but as algebras with involution, is richer than the
Witt-invariant $w$ lying in ${_{2}}{\mathrm{Br}}(\mathcal{O}_{S})$. For
example, if $L$ is an invertible $\mathcal{O}_{S}$-module and
$H(L) = L \oplus L^{*}$ is the corresponding hyperbolic plane, then
$\text{\textbf{C}}(H(L))$ is isomorphic as a graded algebra to $\text{End}
_{\mathcal{O}_{S}}(\wedge L) = \text{End}_{\mathcal{O}_{S}}(
\mathcal{O}_{S}\oplus L)$ (\cite[IV, Prop.~2.1.1]{Knu}) being
Brauer-equivalent to $M_{2}(\mathcal{O}_{S})$, thus $w(H(L))=[0]
\in{\mathrm{Br}}(\mathcal{O}_{S})$, while:
\begin{equation*}
\delta([H(L)]) = \partial([L]) \stackrel{\text{Corollary}~\ref{PictoH2}}{=} \partial^{1}([\text{End}_{\mathcal{O}_{S}}(
\mathcal{O}_{S}\oplus L)]) \in H^{2}_{\text{\'{e}t}}(\mathcal
{O}_{S},\underline{\mu}_{2})
\end{equation*}
(see the left equality in the proof of the Proposition in
\cite[\S5.5]{EKV}), does not have to vanish as we shall see in Proposition~\ref{bijection}.
\end{remark}

\section{A classification of quadratic spaces via their genera}

Consider the ring of $S$-integral ad\`{e}les $\mathbb{A}_{S} :=
\prod_{\mathfrak{p}\in S} \hat{K}_{\mathfrak{p}}\times
\prod_{\mathfrak{p}\notin S} \hat{\mathcal{O}}_{\mathfrak{p}}$, being
a subring of the ad\`{e}les $\mathbb{A}$. Then the $S$-\emph{class set}
of an $\mathcal{O}_{S}$-group $\underline{G}$ is the set of double
cosets:
\begin{equation*}
\text{Cl}_{S}(\underline{G}) := \underline{G}(\mathbb{A}_{S}) \backslash
\underline{G}(\mathbb{A}) / G(K)
\end{equation*}
(where for any prime $\mathfrak{p}$ the geometric fiber $
\underline{G}_{\mathfrak{p}}$ of $\underline{G}$ is taken) and it is
finite (cf. \cite[Prop.~3.9]{BP}). If $\underline{G}$ is affine and
finitely generated over $\text{Spec} \,\mathcal{O}_{S}$, it admits
according to Nisnevich (\cite[Theorem~I.3.5]{Nis}) the following exact
sequence of pointed sets:
%
\begin{equation}
\label{Nissequence}
1 \to\text{Cl}_{S}(\underline{G}) \to H^{1}_{\text{\'{e}t}}(
\mathcal{O}_{S},\underline{G}) \to H^{1}(K,G) \times
\prod_{\mathfrak{p}\notin S} H^{1}_{\text{\'{e}t}}(\hat{\mathcal{O}}_{\mathfrak{p}},(\underline{G})_{\mathfrak{p}}).
\end{equation}
If $\underline{G}$ admits, furthermore, the property:
%
\begin{equation}
\label{locallyembedded}
\forall\mathfrak{p} \notin S:
\ \
H^{1}_{\text{\'{e}t}}(\hat{\mathcal{O}}_{\mathfrak{p}},\underline{G}_{\mathfrak{p}}) \hookrightarrow H^{1}(\hat{K}_{\mathfrak{p}},G_{
\mathfrak{p}}),
\end{equation}
then Nisnevich's sequence for $\underline{G}$ reduces to (cf.
\cite[Corollary~I.3.6]{Nis}, \cite[Corollary~A.8]{GP}):
%
\begin{equation}
\label{Nissimple}
1 \to\text{Cl}_{S}(\underline{G}) \to H^{1}_{\text{\'{e}t}}(
\mathcal{O}_{S},\underline{G}) \to H^{1}(K,G).
\end{equation}

\begin{remark}
\label{finiteetaleextensionisembeddedingenericfiber} Since $\text{Spec}
\,\mathcal{O}_{S}$ is normal, i.e., is integrally closed locally
everywhere (due to the smoothness of $C$), any finite \'{e}tale
covering of $\mathcal{O}_{S}$ arises by its normalization in some
separable unramified extension of $K$ (see \cite[Theorem~6.13]{Len}).
Consequently, if $\underline{G}$ is a finite $\mathcal{O}_{S}$-group,
then $H^{1}_{\text{\'{e}t}}(\mathcal{O}_{S}, \underline{G})$ is
embedded in $H^{1}(K,\underline{G})$. This is not true for infinite
groups like the multiplicative group $\underline{\mathbb{G}}_{m}$, for
which $H^{1}_{\text{\'{e}t}}(\mathcal
{O}_{S},\underline{\mathbb{G}}_{m}) \cong\text{Pic~}(\mathcal{O}_{S})$
clearly does not have to embed in $H^{1}(K,\mathbb{G}_{m})=1$.
\end{remark}

\begin{remark}
\label{genus}
In case $\underline{G}=\underline{\text{\textbf{O}}}_{V}$, the left exactness
of sequence \eqref{Nissequence} reflects the fact that $\text{Cl}
_{S}(\underline{\text{\textbf{O}}}_{V})$ is the \emph{genus} of the base point
$(V,q)$, namely, the set of classes of quadratic $\CS$-forms that are $K$ and
$\hat{\mathcal{O}}_{\mathfrak{p}}$-isomorphic to it for all
$\mathfrak{p} \notin S$. Furthermore, being connected, $
\underline{\text{\textbf{SO}}}_{V}$ admits property \eqref{locallyembedded}
by Lang's Theorem (recall that all residue fields are finite), so the
\emph{proper genus} can be described as:
%
\begin{equation}
\label{propergenus}
\text{Cl}_{S}(\underline{\text{\textbf{SO}}}_{V}) = \ker[H^{1}_{
\text{\'{e}t}}(\mathcal{O}_{S},\underline{\text{\textbf{SO}}}_{V}) \to H^{1}(K,
\text{\textbf{SO}}_{V})].
\end{equation}
As $\underline{\text{\textbf{O}}}_{V}/\underline{\text{\textbf{SO}}}_{V}$ is the
finite representable $\mathcal{O}_{S}$-group $\underline{\mu}_{2}$ (cf.
\cite[Theorem~1.7]{Con}), $\underline{\text{\textbf{O}}}_{V}$ admits property \eqref{locallyembedded} as well (see in the proof of Proposition~3.4
in \cite{CGP}), so we may also write:
%
\begin{equation}
\text{Cl}_{S}(\underline{\text{\textbf{O}}}_{V}) = \ker[H^{1}_{
\text{\'{e}t}}(\mathcal{O}_{S},\underline{\text{\textbf{O}}}_{V}) \to H^{1}(K,
\text{\textbf{O}}_{V})].
\end{equation}
As a pointed set, $\text{Cl}_{S}(\underline{\text{\textbf{SO}}}_{V})$ is
bijective to the first Nisnevich cohomology set $H^{1}_{\text{Nis}}(
\mathcal{O}_{S},\underline{\text{\textbf{SO}}}_{V})$ (cf.
\cite[Theorem~I.2.8]{Nis} and \cite[4.1]{Mor}), classifying
$\underline{\text{\textbf{SO}}}_{V}$-torsors in the Nisnevich topology. But
Nisnevich covers are \'{e}tale, so it is a subset of $H^{1}_{
\text{\'{e}t}}(\mathcal{O}_{S},\underline{\text{\textbf{SO}}}_{V})$. Similarly,
$\text{Cl}_{S}(\underline{\text{\textbf{O}}}_{V}) \subseteq H^{1}_{
\text{\'{e}t}}(\mathcal{O}_{S},\underline{\text{\textbf{O}}}_{V})$.
\end{remark}

\begin{lem}
\label{isotropicreflection}
If $(V,q)$ is isotropic then $\underline{\text{\textbf{O}}}_{V}(\mathcal{O}
_{S}) \xrightarrow{\det} \underline{\mu}_{2}(\mathcal{O}_{S})$ is
surjective.
\end{lem}

\begin{proof}
Consider the following exact and commutative diagram that arises by
applying \'{e}tale cohomology to the short exact sequences related to
the smooth $\mathcal{O}_{S}$-groups $\underline{\text{\textbf{Pin}}}_{V}$ and
$\underline{\text{\textbf{O}}}_{V}$:
\begin{equation*}
\xymatrix{
& & H^1_\text{\'et}(\mathcal{O}_S,\underline{\text{\textbf{Spin}}}_V) \ar[r] \ar[d]^{s\pi_*} & H^1_\text{\'et}
(\mathcal{O}_S,\underline{\text{\textbf{Pin}}}_V) \ar[d]^{\pi_*} \\
\underline{\text{\textbf{O}}}_{V}(\mathcal{O}_S) \ar[r]^{\det} &\underline{\mu}_2(\mathcal{O}_S) \ar[r]^{\partial_0} & H^1_\text{\'et}
(\mathcal{O}_S,\underline{\text{\textbf{SO}}}_{V}) \ar[r]^{h} \ar[d]^{s\delta_V} & H^1_\text{\'et}(\mathcal{O}_S,\underline{\text{\textbf{O}}}_{V}) \ar[d]^{\delta_V}
\\
& & H^2_\text{\'et}(\mathcal{O}_S,\underline{\mu}_2) \ar@{=}[r] & H^2_\text{\'et}(\mathcal{O}_S,\underline{\mu}_2).
}
\end{equation*}
Denote $[\gamma] = \partial_{0}(-1)$. Then $s\delta_{V}([\gamma]) =
\delta_{V}(h([\gamma])=[0])=[0]$, hence $[\gamma] \in
\operatorname{Im}(s\pi_{*})$. But as $q$ is isotropic, $H^{1}_{
\text{\'{e}t}}(\mathcal{O}_{S},\underline{\text{\textbf{Spin}}}_{V})$ vanishes
by strong approximation (cf. Lemma~\ref{H1=1sc}), so $[\gamma] = [0]$,
which means that $\partial_{0}$ is the trivial map and $\det(
\mathcal{O}_{S})$ surjects on $\underline{\mu}_{2}(\mathcal{O}_{S})$.
\end{proof}

\begin{lem}
\label{samegenus}
If $n$ is odd, or $(V,q)$ is isotropic, then $\text{Cl}_{S}(\underline{\text{\textbf{SO}}}
_{V}) = \text{Cl}_{S}(\underline{\text{\textbf{O}}}_{V})$.
\end{lem}

\begin{proof}
Any representative $(V',q')$ of a class in
$\text{Cl}_{S}(\underline{\text{\textbf{O}}}_{V})$, being $K$
isomorphic to $q$, is regular and isotropic as well, whence
$\underline{\text{\textbf{O}}}_{V'}(\mathcal{O}_{S}) \to
\underline{\mu}_{2}(\mathcal{O}_{S})$ is surjective by Lemma~\ref{isotropicreflection}. 
When $n$ is odd, this surjectivity is retrieved by the fact that $\underline{\text{\textbf{O}}}_{V'} \cong \underline{\text{\textbf{SO}}}_{V'} \times \underline{\mu}_2$ (cf. \cite[Thm.1.7]{Con}), and so applying \'{e}tale cohomology to the
exact sequence of smooth groups
\begin{equation*}
1 \to\underline{\text{\textbf{SO}}}_{V'} \to\underline{\text{\textbf{O}}}_{V'}
\to\underline{\mu}_{2} \to1
\end{equation*}
we get that
$\ker[H^{1}_{\text{\'{e}t}}(\mathcal{O}_{S},\underline{\text{\textbf{SO}}}_{V'})
\xrightarrow{\psi'} H^{1}_{\text{\'{e}t}}(
\mathcal{O}_{S},\underline{\text{\textbf{O}}}_{V'})]=1$ for any
$[(V',q')] \in\text{Cl}_{S}(\underline{\text{\textbf{O}}}_{V})$, which
means that the restricted map
$\text{Cl}_{S}(\underline{\text{\textbf{SO}}}_{V}) \xrightarrow{ \psi}
\text{Cl}_{S}(\underline{\text{\textbf{O}}}_{V})$ is injective.
Together with Remark~\ref{genus}, this amounts to the existence of the
following exact and commutative diagram:
\begin{equation*}
\xymatrix{
& \text{Cl}_S(\underline{\text{\textbf{SO}}}_{V}) \ar@{^{(}->}[r]^{\psi} \ar@{^{(}->}[d]^{i} & \text{Cl}_S(\underline{\text{\textbf{O}}}_{V}) \ar
@{^{(}->}[d]^{i'} \\
1 \ar[r] & H^1_\text{\'et}(\mathcal{O}_S,\underline{\text{\textbf{SO}}}_{V}) \ar[r]^{\psi'} \ar[d]^{m} & H^1_\text{\'et}(\mathcal{O}_S,\underline{\text{\textbf{O}}}_{V}
) \ar[d]^{m'} \ar[r]^{d} & H^1_\text{\'et}(\mathcal{O}_S,\underline{\mu}_2) \ar
@{^{(}->}[d]^{m''} \\
1 \ar[r] & H^1(K,\text{\textbf{SO}}_V) \ar@{^{(}->}[r]^{h} & H^1(K,\textbf
{O}_V) \ar[r]^{d'} & H^1(K,\mu_2)
}
\end{equation*}
in which as $m''$ is injective due to Remark~\ref{finiteetaleextensionisembeddedingenericfiber}, 
$\psi$ is also surjective, thus is the identity.
\end{proof}

\begin{prop}
\label{Wittexactness}
Let $(V,q)$ be a regular quadratic $\mathcal{O}_{S}$-space of rank
$n \geq3$ with proper genus $\text{Cl}_{S}(\underline{\text{\textbf{SO}}}
_{V})$. The relative Witt-invariant (cf. Definition~\ref{relativeWitt-invariant} and Remark~\ref{restriction}) induces an
exact sequence of pointed sets
\begin{equation*}
1 \to\text{Cl}_{S}(\underline{\text{\textbf{SO}}}_{V}) \xrightarrow{h} H
^{1}_{\text{\'{e}t}}(\mathcal{O}_{S},\underline{\text{\textbf{SO}}}_{V})
\xrightarrow{w
_{V}} {_{2}{\mathrm{Br}}(\mathcal{O}_{S})} \to1
\end{equation*}
in which $h$ is injective and ${_{2}}{\mathrm{Br}}(\mathcal{O}_{S})$
bijects to the set of $2^{|S|-1}$ proper genera of $q$.
\end{prop}

\begin{proof}
Consider the short exact sequence induced by the double covering of the
generic fiber
\begin{equation*}
1 \to\mu_{2} \to\text{\textbf{Spin}}_{V} \to\text{\textbf{SO}}_{V} \to1.
\end{equation*}
As $\text{\textbf{Spin}}_{V}$ is simply connected, we know due to
Harder that $H^{1}(K,\text{\textbf{Spin}}_{V})=1$ (cf.
\cite[Satz~A]{Hard}). This is true for all twisted forms of $\text{\textbf{Spin}}_{V}$, whence Galois cohomology implies the embedding
$H^{1}(K,\text{\textbf{SO}}_{V}) \hookrightarrow H^{2}(K,\mu_{2})$. Due
to Hilbert's Theorem~90, applying Galois cohomology to the Kummer's
exact sequence related to $\mu_{2}$ over $K$ gives an isomorphism
$H^{2}(K,\mu_{2}) \cong{_{2} {\mathrm{Br}}(K)}$. Moreover, as shown in
the sequence \eqref{Grothendieck}, ${\mathrm{Br}}(\mathcal{O}_{S})$ is
a subgroup of ${\mathrm{Br}}(K)$. All together, the relative
Witt-invariant applied to classes in
$H^{1}_{\text{\'{e}t}}(\mathcal{O}_{S},\underline{\text{\textbf{SO}}}_{V})$
and on their generic fibers, yields the following exact and commutative
diagram:
%
\begin{equation}
\label{Wittdiagram}
\xymatrix{
H^1_\text{\'et}(\mathcal{O}_S,\underline{\text{\textbf{SO}}}_{V}) \ar[r]^{w_V} \ar[d] & {_2 {\mathrm{Br}}(\mathcal{O}_S)} \ar@{^{(}->}[d] \\
H^1(K,\text{\textbf{SO}}_V) \ar@{^{(}->}[r]^{w_V} & {_2 {\mathrm{Br}}(K)}
}
\end{equation}
which justifies the left exactness in the asserted sequence:
%
\begin{eqnarray}
\label{genusinjection}
\text{Cl}_{S}(\underline{\text{\textbf{SO}}}_{V}) \stackrel{\text{(\ref{propergenus})}}{=} \ker[H^{1}_{\text{\'{e}t}}(\mathcal
{O}_{S},\underline{\text{\textbf{SO}}}_{V}) \to
H^{1}(K,\text{\textbf{SO}}_{V})] = \ker[H^{1}_{
\text{\'{e}t}}(\mathcal{O}_{S},\underline{\text{\textbf{SO}}}_{V})
\xrightarrow{w _{V}} {_{2} {\mathrm{Br}}(\mathcal{O}_{S})}].\nonumber\\
\end{eqnarray}
The surjectivity of $w_{V}:H^{1}_{\text{\'{e}t}}(\mathcal
{O}_{S},\underline{\text{\textbf{SO}}}_{V})
\to{_{2}{\mathrm{Br}}(\mathcal{O}_{S})}$ (cf. Remark~\ref{restriction})
completes the proof.

The first equality in \eqref{genusinjection} suggests that for any
$[q'] \in
H^{1}_{\text{\'{e}t}}(\mathcal{O}_{S},\underline{\text{\textbf{SO}}}_{V})$,
$q' \in\text{Cl}_{S}(\underline{\text{\textbf{SO}}}_{V})$ if and only
if $q'$ is $K$-isomorphic to $q$ while the second equality claims that
this holds if and only if $w_{V}(q')=w_{V}(q)$. This is true for any
choice of base point, being regular as well. So as $H^{1}_{
\text{\'{e}t}}(\mathcal{O}_{S},\underline{\text{\textbf{SO}}}_{V})$ is
a disjoint union of the proper genera of $q$, together with the
surjectivity of $w_{V}$, we deduce that the set of proper genera
bijects with $_{2} {\mathrm{Br}}(\mathcal{O}_{S})$, whose cardinality
is computed in Lemma~\ref{no.genera}. 
\end{proof}

\begin{theorem}
\label{isotropicclassification}
Let $(V,q)$ be a regular quadratic $\mathcal{O}_{S}$-space
of rank $n \geq3$. Then there exists a surjection of pointed-sets: 
${Cl}_{S}(\underline{\text{\textbf{SO}}}_{V}) \twoheadrightarrow \text{Pic~}(\mathcal{O}_{S})/2$. 
If $(V,q)$ is isotropic, then this is a bijection, 
so the abelian group $\text{Pic~}(\mathcal{O}_{S})/2$ is isomorphic to ${Cl}_{S}(\underline{\text{\textbf{SO}}}_{V}) = 
 {Cl}_{S}(\underline{\text{\textbf{O}}}_{V})$. 
\end{theorem}


\begin{proof}
Consider the following exact and
commutative diagram derived from sequences \eqref{Kummermu2H1} and \eqref{H1Spin}:
%
\begin{align}
\label{Snakediagram}
\xymatrix{
& 1 \ar[r] \ar[d] & H^1_\text{\'et}(\mathcal{O}_S,\underline{\text{\textbf{SO}}}_{V}) \ar@{=}[r] \ar@{->>}[d]^{s\delta_V}
& H^1_\text{\'et}(\mathcal{O}_S,\underline{\text{\textbf{SO}}}_{V}) \ar@{->>}[d]^{w_V} \ar[r] & 1 \\
1 \ar[r] & \text{Pic~}(\mathcal{O}_S)/2 \ar[r]^{\partial} & H^2_\text{\'et}(\mathcal{O}_S,\underline{\mu}_2) \ar
[r]^{i_*^2} & {_2 {\mathrm{Br}}(\mathcal{O}_S)} \ar[r] & 1.
}
\end{align}
We imitate the Snake Lemma argument (though the diagram terms are not
all groups): according to Proposition~\ref{Wittexactness} $\text{Cl}
_{S}(\underline{\text{\textbf{SO}}}_{V}) = \ker(w_{V})$, hence for any
$[q'] \in\text{Cl}_{S}(\underline{\text{\textbf{SO}}}_{V})$ one has
$i_{*}^{2}(s\delta_{V}([q']))=[0]$, i.e., $[q']$ has a
$\partial$-preimage in $\text{Pic~}(\mathcal{O}_{S})/2$ which is unique
as $\partial$ is a monomorphism of groups. Moreover, any element in
$\text{Pic~}(\mathcal{O}_{S})/2$ arises in this way, since $s\delta
_{V}$ is surjective. As a result we have an exact sequence of pointed
sets:
\begin{equation*}
1 \to\mathfrak{K}_{1} \to\text{Cl}_{S}(\underline{\text{\textbf{SO}}}_{V})
\xrightarrow{s\delta_{V}} \text{Pic~}(\mathcal{O}_{S})/2 \to1.
\end{equation*}
If $(V,q)$ is isotropic, then $s\delta_{V}$ is also injective. 
Indeed, 
let $\underline{\text{\textbf{SO}}}_{V'}$ be a twisted form of $\underline{\text{\textbf{SO}}}_{V}$, properly
stabilizing a form $q'$, and let $\underline{\text{\textbf{Spin}}}_{V'}$ be its
Spin group. The lower row in the following exact diagram is the one
obtained when replacing the base point $q$ by $q'$, 
as described in \cite[IV, Proposition~4.3.4]{Gir}:
\begin{equation} \label{Gir}
\xymatrix{
        \mathfrak{K}_1             \ar[r] & \ClS(\SOV)         \ar@{}[d]|-*[@]{\subset} \ar@{->>}[r]^{s\dl_V}   & \Pic(\CS) / 2  \ar[d]  \\
H^1_\et(\CS,\un{\textbf{Spin}}_V)  \ar[r] & H^1_\et(\CS,\SOV)  \ar@{->>}[r]^{s\dl_V}                            & H^2_\et(\CS,\un{\mu}_2)    \\
H^1_\et(\CS,\un{\textbf{Spin}}_{V'}) \ar[r] & H^1_\et(\CS,\SOV') \ar@{->>}[r]^{\dl_{V'}} \ar[u]_{\cong}^{\theta}  & H^2_\et(\CS,\un{\mu}_2) \ar[u]^{r}_{\cong}  
}
\end{equation} 
If $[q'] \in\text{Cl}_{S}(\underline{\text{\textbf{SO}}}_{V})$, then
$q'$, being $K$-isomorphic to $q$, is also $K$-isotropic, as well as
the generic fiber $\text{\textbf{Spin}}_{V'}$. Then by the
Hasse--Minkowsky Theorem (cf. \cite[VI.3.1]{Lam}), $q'$ is
$\hat{K}_{\mathfrak{p}}$-isotropic everywhere, in particular in $S$.
Hence $H^{1}_{\text{\'{e}t}}(
\mathcal{O}_{S},\underline{\text{\textbf{Spin}}}_{V'})$ is trivial by Lemma~\ref{H1=1sc} for any class in
$\text{Cl}_{S}(\underline{\text{\textbf{SO}}} _{V})$
($\underline{\text{\textbf{SO}}}_{V}$ is not commutative for $n \geq3$
thus
$H^{1}_{\text{\'{e}t}}(\mathcal{O}_{S},\underline{\text{\textbf{SO}}}_{V})$
does not have to be a group, so the triviality of
$H^{1}_{\text{\'{e}t}}(\mathcal{O}_{S},\underline{\text{\textbf{Spin}}}_{V})$
does not imply the injectivity of $s\delta_{V}$, i.e., there still
might be distinct anisotropic classes in
$H^{1}_{\text{\'{e}t}}(\mathcal{O}
_{S},\underline{\text{\textbf{SO}}}_{V})$ whose images in $H^{2}_{
\text{\'{e}t}}(\mathcal{O}_{S},\underline{\mu}_{2})$ coincide).  
So $\text{Cl}_{S}(\underline{\text{\textbf{SO}}}_{V})$, being equal in the isotropic case to $\text{Cl} _{S}(\underline{\text{\textbf{O}}}_{V})$
by Lemma~\ref{samegenus}, embeds in $\text{Pic~}(\mathcal{O}_{S})/2$
and the assertion follows.
\end{proof}

\begin{definition}
We say that the \emph{local-global Hasse principle} holds for a
quadratic $\mathcal{O}_{S}$-space $(V,q)$ if
$|\text{Cl}_{S}(\underline{\text{\textbf{O}}}_{V})|=1$.
\end{definition}

\begin{cor}
The Hasse principle holds for a regular isotropic quadratic
$\mathcal{O}_{S}$-form of rank $\geq3$ if and only if $|\text{Pic~}(
\mathcal{O}_{S})|$ is odd.
\end{cor}

\begin{cor}
\label{genera} Let $(V,q)$ be an $\mathcal{O}_{S}$-regular quadratic
space of rank $\geq5$. Then the pointed-set 
$H^{1}_{\text{\'{e}t}}(\mathcal{O}_{S},\underline{\textbf{SO}}_{V})$
is isomorphic to the abelian group $H^{2}_{
\text{\'{e}t}}(\mathcal{O}_{S},\underline{\mu}_{2})$, 
i.e., any $\mathcal{O}_{S}$-isomorphism class in the proper classification corresponds to an Azumaya $\mathcal{O}_{S}$-algebra with involution. There are $2^{|S|-1}$ genera, each of them is isomorphic to $\text{Pic~}(\mathcal{O}_{S})/2$.
\end{cor}

\begin{proof}
In rank $\geq5$, any quadratic $\mathcal{O}_{S}$-space is isotropic;
indeed, for any such $(V',q')$, the generic fiber $q'_{K}:= q' \otimes
K$ is isotropic (cf. \cite[Theorem~66:2]{OMe}), i.e., there exists a
non-zero vector $v_{0} \in V' \otimes K$ such that $q'_{k}(v_{0})=0$.
Since $K$ is the fraction field of the Dedekind domain $\mathcal{O}
_{S}$, there exists a non-zero vector $\underline{v}_{0} \in Kv_{0}
\cap\mathcal{O}_{S}$ for which $q(\underline{v}_{0})=0$. Hence
according to Theorem~\ref{isotropicclassification}, we deduce that the
genus of any $[(V',q')] \in H^{1}_{\text{\'{e}t}}(\mathcal
{O}_{S},\underline{\text{\textbf{SO}}}_{V})$ is isomorphic to the
abelian group $\text{Pic~}( \mathcal{O}_{S})/2$ and injects into
$H^{1}_{\text{\'{e}t}}(
\mathcal{O}_{S},\underline{\text{\textbf{SO}}}_{V})$. Looking at the
obtained exact and commutative diagram:
\begin{equation*}
\xymatrix{
1 \ar[r] & \text{Cl}_S(\underline{\text{\textbf{SO}}}_{V}) \ar[r] \ar[d]^{\cong} & H^1_\text{\'et}(\mathcal{O}_S,\underline{\text{\textbf{SO}}}_{V}) \ar
[r]^{w_V} \ar@{->>}[d]^{s\delta_V} & {_2 {\mathrm{Br}}(\mathcal{O}_S)} \ar[r] \ar@{=}[d] & 1 \\
1 \ar[r] & \text{Pic~}(\mathcal{O}_S)/2 \ar[r]^{\partial} & H^2_\text{\'et}(\mathcal{O}_S,\underline{\mu}_2) \ar
[r]^{i_*^2} & {_2 {\mathrm{Br}}(\mathcal{O}_S)} \ar[r] & 1
}
\end{equation*}
we see that the cardinality of $H^{1}_{\text{\'{e}t}}(\mathcal
{O}_{S},\underline{\text{\textbf{SO}}}_{V})$, being the disjoint union
of its genera, equals the one of its $s\delta_{V}$-image
$H^{2}_{\text{\'{e}t}}(\mathcal {O}_{S},\underline{ \mu}_{2})$, hence
it is isomorphic to it. We have seen in Proposition~\ref{Wittexactness} 
that ${_{2}}{\mathrm{Br}}(\mathcal{O}_{S})$ bijects with the set of
proper genera of $q$. Here, as any twisted form of $q$ is isotropic,
its proper genus is equal to its genus (Lemma~\ref{samegenus}), thus
there are $2^{|S|-1}$ such genera (Lemma~\ref{no.genera}). 
\end{proof}

\begin{remark}
Since as we have seen any integral quadratic form of rank $\geq 5$ is isotropic, according to Lemma \ref{samegenus}, 
$\text{Cl}_S(\underline{\text{\textbf{SO}}}_{V})$ might not be equal to $\text{Cl}_S(\underline{\text{\textbf{O}}}_{V})$ 
(for $\text{rank}(V) \geq 3$) only when $(V,q)$ is anisotropic of rank $4$. 
\end{remark}

\section{A splitting hyperbolic plane}\label{hyperbolicplane}

In this section, we refer to regular quadratic $\mathcal{O}_{S}$-spaces
being split by a hyperbolic plane $P$ (thus being isotropic and so
$\text{Cl}_{S}(\underline{\text{\textbf{O}}}_{V}) =
\text{Cl}_{S}(\underline{\text{\textbf{SO}}}_{V})$). Such $P$ contains
a hyperbolic pair $\{v_{0},v _{1}\}$ of $V$, namely, satisfying:
$q(v_{0})=q(v_{1})=0$ and $B_{q}(v_{0},v_{1})=1$, and it is of the form
$P_{\mathfrak{a}}= \mathfrak{a}v_{0} + \mathfrak{a}^{-1} v_{1}$ for
some fractional ideal $\mathfrak{a}$ of $\mathcal{O}_{S}$ (cf.
\cite[Proposition~2.1]{Ger}). L. J. Gerstein established in
\cite[Theorem~4.5]{Ger} for the ternary case the bijection:
\begin{equation*}
\psi: \text{Pic~}(\mathcal{O}_{S})/2 \xrightarrow{\sim} \text{Cl}
_{S}(\underline{\text{\textbf{O}}}_{V}): \ \ [\mathfrak{a}] \mapsto[V_{
\mathfrak{a}}] := [P_{\mathfrak{a}}\bot\mathfrak{b}\lambda]
\end{equation*}
where $\mathfrak{b}$ is again a fractional ideal of $\mathcal{O}_{S}$
and $\lambda\in\mathcal{O}_{S}^{\times}$, for which $V \cong V_{
\mathfrak{a}}$. The existence of $\psi$ can be viewed as a particular
case of our Theorem~\ref{isotropicclassification} (such a splitting
does not necessarily exist). The following Lemma suggests an alternative
bijection which resembles the one of Gerstein, but does not, however,
require finding and multiplying by a hyperbolic pair, and, more
important, is valid for any rank $n \geq3$.

\begin{prop}
\label{bijection} Suppose $V$ is split by a hyperbolic plane
$V=H(L_{0}) \bot V_{0}$ where $L_{0}$ is an $\mathcal{O}_{S}$-line.
Then $\psi_{V}: [L] \mapsto[V _{L} = H(L_{0} \otimes L^{*}) \bot
V_{0}]$ is an isomorphism of groups $\text{Pic~}(\mathcal{O}_{S})/2
\cong\text{Cl}_{S}(\underline{\text{\textbf{O}}}_{V})$.
\end{prop}

\begin{proof}
$L^{*}$ is locally free, thus $V_{L}$ obtained from $V$ by tensoring
$H(L_{0})$ with $L^{*}$ remains in
$\text{Cl}_{S}(\underline{\text{\textbf{O}}}_{V})$. Due to Theorem~\ref{isotropicclassification}, the groups
$\text{Pic~}(\mathcal{O}_{S})/2$ and
$\text{Cl}_{S}(\underline{\text{\textbf{O}}}_{V})$ are isomorphic
through diagram \eqref{Snakediagram}, so it is sufficient to show by
its commutativity that $s\delta_{V} \circ\psi_{V}$ coincides with the
groups embedding $\partial$. The $i$'th \emph{Stiefel--Whitney class}
$w_{i}(E)$ of a regular $\mathcal{O}_{S}$-module $E$, as defined in
\cite[\S1]{EKV}, gets values in
$H^{i}_{\text{\'{e}t}}(\mathcal{O}_{S},\underline{ \mu}_{2})$ for $i
\geq1$ and \mbox{$w_{0}(E)=1$}. Its basic axioms, namely, $w_{i}(E)=0$ for all
$i > \text{rank}(E)$ and for any direct sum of regular
$\mathcal{O}_{S}$-modules of finite rank
\begin{equation*}
w_{k}(E \oplus F) = \sum_{i+j=k} w_{i}(E) \cdot w_{j}(F),
\end{equation*}
imply that:
%
\begin{equation}
\label{w2}
w_{1}(E \oplus F) = w_{1}(E) + w_{1}(F)
\ \
\text{and:}
\ \
w_{2}(E \oplus F) = w_{2}(E) + w_{2}(F) + w_{1}(E) \cdot w_{1}(F).
\end{equation}
If $L$ is an $\mathcal{O}_{S}$-line and $H(L) = L \oplus L^{*}$ is the
corresponding hyperbolic plane, this reads:
\begin{equation*}
w_{1}(H(L)) = w_{1}(L) + w_{1}(L^{*})
\ \
\text{while}
\ \
w_{2}(H(L)) = w_{1}(L) \cdot w_{1}(L^{*}).
\end{equation*}
Moreover, $w_{1}$ furnishes an isomorphism of abelian groups
$\{ \mathcal{O}_{S}\text{-lines}, \otimes\}/\sim\cong H^{1}_{
\text{\'{e}t}}(\mathcal{O}_{S},\underline{\mu}_{2})$ by
$$ w_1(L_1 \otimes L_2) = w_1(L_1) + w_1(L_2), \ \ \text{thus:} \ \ w_1(L) = - w_1(L^*),  $$
hence
\begin{align} \label{w1} 
w_1(H(L_0 \otimes L^*)) - w_1(H(L_0)) &= w_1(L_0 \otimes L^*) + w_1(L_0^* \otimes L) - w_1(L_0) - w_1(L_0^*) \\ \nonumber
                                      &= w_1(L_0) + w_1(L^*) + w_1(L_0^*) + w_1(L) - w_1(L_0) - w_1(L_0^*) = 0 
\end{align}
and:
\begin{align} \label{diffofw2}
w_2(H(L_0 \otimes L^*)) - w_2(H(L_0)) &= (w_1(L_0) + w_1(L^*)) \cdot (w_1(L_0^*) + w_1(L)) - w_1(L_0) \cdot w_1(L_0^*) \\ \nonumber
                                      &= w_1(L^*) \cdot w_1(L_0^*) + w_1(L_0) \cdot w_1(L) + w_1(L^*) \cdot w_1(L) \\ \nonumber
                                      &= w_1(L^*) \cdot w_1(L_0^*) - w_1(L_0^*) \cdot w_1(L) + w_1(L^*) \cdot w_1(L) \\ \nonumber
                                      &= w_1(L^*) \cdot w_1(L) = w_2(H(L)).   
\end{align}
In our setting (recall that $\delta=w_{2}$, see Remark~\ref{Stiefel-Whitney}), we get:
%
\begin{align}
\label{diffofdl}
\delta([V_{L}]) - \delta([V])
&\stackrel{\text{(\ref{w2})}}{=} \delta([H(L
_{0} \otimes L^{*})]) + \delta([V_{0}]) + w_{1}(H(L_{0} \otimes L
^{*})) \cdot w_{1}(V_{0})
\\
\nonumber
& \ -\Big( \delta([H(L_{0})]) + \delta([V_{0}]) + w_{1}(H(L_{0}))
\cdot w_{1}(V_{0}) \Big)
\\
\nonumber
&\stackrel{\text{(\ref{w1})}}{=} \delta([H(L_{0} \otimes L^{*})]) - \delta
([H(L_{0})]) \stackrel{\text{(\ref{diffofw2})}}{=} \delta([H(L)]).
\end{align}
Altogether, we may finally conclude that for any $[L] \in\text{Pic~}(
\mathcal{O}_{S})/2$:
\begin{align*}
(s\delta_{V} \circ\psi_{V})([L])
&= s\delta_{V}([V_{L}]) \stackrel{\text{(\ref{relativew})}}{=} (r_{V} \circ\delta\circ\theta^{-1})([V_{L}])
= (r_{V} \circ\delta)([V_{L}])
\\
\nonumber
&= \delta([V_{L}])-\delta([V]) \stackrel{\text{(\ref{diffofdl})}}{=}
\delta([H(L)]) \stackrel{\text{Remark~\ref{projectiveplane1}}}{=}
\partial([L]). \qedhere
\end{align*}
\end{proof}

\subsection{$|S|=1$}

If $S=\{\infty\}$ where $\infty$ is an arbitrary closed point, then
$C^{\text{af}}:= C -\{\infty\}$ is an affine curve whence ${\mathrm{Br}}(
\mathcal{O}_{S})=1$ (cf. Lemma~\ref{no.genera}), i.e., there is a
single genus, and any $\mathcal{O}_{S}$-regular quadratic space
$(V,q)$ of rank $n \geq3$ is isotropic (cf. \cite{Bit}). Suppose
furthermore that $C$ is an elliptic curve and $\infty$ is $
\mathbb{F}_{q}$-rational. For any place $P$ on $C^{\text{af}}$ we define
the maximal ideal of $\mathcal{O}_{S}$
\begin{equation*}
\mathfrak{m}_{P} := \left\{ f \in\mathcal{O}_{S}\ : \ f(P) = 0
\right\} .
\end{equation*}
Then we have an isomorphism of abelian groups (see
\cite[II,~Proposition~6.5(c)]{Hart} for the first map and p.~393 in
\cite{Bau} for the second one):
\begin{equation*}
\varphi: C(\mathbb{F}_{q}) \cong\text{Pic~}^{0}(C) \cong
\text{Pic~}(\mathcal{O}_{S}): \ \ P \mapsto[P] - [\infty] \mapsto
\mathfrak{m}_{P}.
\end{equation*}
Since $\mathcal{O}_{S}$ is Dedekind, Steinitz's Theorem (cf. \cite[Corollary~6.1.9]{BK}
tells us that for any fractional ideal
$\mathfrak{a}$ of $\mathcal{O}_{S}$ there is an
$\mathcal{O}_{S}$-isomorphism of $\mathcal{O}_{S}$-modules 
(though not of quadratic $\mathcal{O}_{S}$-spaces):
$$ \chi: \fa \oplus \fa^{-1} \to \CS \oplus \CS. $$ 
Towards our purpose of finding representatives of twisted $K$-forms of
$q$, we may choose such a $K$-isomorphism $\chi_{K}$ to be as in
\cite[Lemma~20.17]{Cla}:
\begin{equation*}
\chi_{K}:(\alpha,\beta) \mapsto(\alpha,\beta) \cdot A_{
\mathfrak{a}}, \ \ A_{\mathfrak{a}}=
\left(
\begin{array}{c@{\quad}c}
b_{1} & -a_{2}
\\
b_{2} & a_{1}
\end{array}
\right)
\end{equation*}
where $a_{1}b_{1} + a_{2}b_{2}=1, \ b_{1} \in\mathfrak{a}^{-1}, \ b
_{2} \in\mathfrak{a}$. According to Proposition~\ref{bijection}, the
matrix $A_{\mathfrak{m}_{P}}^{-t} B_{q} A_{\mathfrak{m}_{P}}^{-1}$
represents, for any coset $[P] \in C(\mathbb{F}_{q})/2$, a distinct
class of quadratic $\mathcal{O}_{S}$-forms in
$\text{Cl}_{S}(\underline{\text{\textbf{O}}}_{V}) =
H^{1}_{\text{\'{e}t}}(\mathcal{O}_{S},\underline{\text{\textbf{SO}}}_{V})$.

We summarize this procedure by the following algorithm:

\begin{algorithm} 
\caption{Generator of classes representatives isomorphic in the flat topology to a regular quadratic space split by a hyperbolic plane,  
over the coordinate ring of an affine non-singular elliptic curve} \label{Generatoralgorithm}
\textbf{Input:} $C$ = elliptic projective $\BF_q$-curve, $S=\{\iy \in C(\BF_q) \}$, 
$V = H(L_0) \bot V_0$ = quadratic regular $\CS$-space of rank $n \geq 3$, 
$H(L_0)$ is represented by $F_0 \in \textbf{GL}_2(\CS)$.  
\begin{algorithmic} 
\State Compute $C(\BF_q)$ and 
$\CS:=C^\af(\BF_q)$ where $C^\af := C-\{\iy\}$.  
\For{each $[P] \in C(\BF_q)/2$ }
\State $\fm_P = \{  f \in \CS \ : \ f(P) = 0 \}$ 
for $P \neq \iy$ and $\fm_{\iy}=\CS$. 
\State Find $a_1,b_2 \in \fm_P$ and $a_2,b_1 \in \fm_P^{-1}$ such that $a_1 b_1 + a_2 b_2 =1$.  
\State  
$A_{\fm_P}^{-1} = \left( \begin{array}{cc}
    a_1  & a_2  \\ 
    -b_2 & b_1  \\ 
\end{array}\right)$ 
and: $F_P = A_{\fm_P}^{-t} F_0 A_{\fm_P}^{-1}$.
\EndFor
\State \textbf{Output:} $\ClS(\OV) = \{ [V_P \bot V_0] : [P] \in C(\BF_q)/2 \}$, where $V_P$ is the quadratic $\CS$-form represented by $F_P$.  
\end{algorithmic}
\end{algorithm}

\begin{remark}
If $C$ is an elliptic curve and $\infty$ is $\mathbb{F}_{q}$-rational,
then for any $\mathcal{O}_{S}$-line $L_{0}$, the special orthogonal
group of $H(L_{0})$, being split, of rank $2$ and
$\mathcal{O}_{S}$-regular, is a one dimensional split $\mathcal{O}_{S}$-torus, i.e., isomorphic to $\underline{\mathbb{G}}_{m}$ (see in
the proof of \cite[Theorem~4.5]{Bit}), hence the proper classification
is via $H^{1}_{\text{\'{e}t}}(\mathcal{O}_{S},\underline{\mathbb{G}}
_{m}) \cong\text{Pic~}(\mathcal{O}_{S}) \cong C(\mathbb{F}_{q})$. Then
the class representatives are obtained by the above algorithm when replacing $C(\mathbb{F}_{q})/2$ by $C(\mathbb{F}_{q})$. This means that invertible fractional ideals, corresponding to non-trivial squares in
$C(\mathbb{F}_{q})$, i.e., $[L] \in2 \text{Pic~}(\mathcal{O}_{S})
\backslash\{0\}$, and only they, induce spaces $H(L_{0} \otimes L
^{*})$ that are \emph{stably isomorphic} to $H(L_{0})$ in the proper classification, namely, become
properly isomorphic after being extended by any non-trivial regular orthogonal
$\mathcal{O}_{S}$-space. In other words, the Witt Cancellation Theorem
fails over $\mathcal{O}_{S}$ in this case for the proper classification.
\end{remark}

\begin{example}\footnote{I take this opportunity to correct Example~4.12 in
\cite{Bit}.} Let $C = \{ Y^{2}Z = X^{3} + XZ^{2}\}$ defined over
$\mathbb{F}_{5}$. Then
\begin{equation*}
C(\mathbb{F}_{5}) = \{ (0:0:1), (1:0:2), (1:0:3), (0:1:0) \}.
\end{equation*}
Taking $S = \{\infty= (0:1:0)\}$ we get the affine elliptic curve
\begin{equation*}
C^{\text{af}}= \{y^{2} = x^{3} + x \} \ \ \text{with:} \ \
\mathcal{O}_{S}= \mathbb{F}_{5}[x,y]/(y^{2}-x^{3}-x).
\end{equation*}
The affine supports of the points in $C(\mathbb{F}_{5}) - \{ \infty
\}$ are: $ \{ (0,0), (1/2,0) = (3,0), (1/3,0) = (2,0) \}$. The
$y$-coordinate of these points vanishes which means that they are of
order $2$ according to the group law and $C(\mathbb{F}_{q}) \cong(
\mathbb{Z}/2)^{2}$. We get:
\begin{equation*}
\mathfrak{m}_{(0:0:1)} = \langle x,y \rangle, \ \mathfrak{m}_{(1:0:2)}
= \langle x-3,y \rangle, \ \mathfrak{m}_{(1:0:3)} = \langle x-2,y
\rangle, \ \mathfrak{m}_{(0:1:0)} = \mathcal{O}_{S}.
\end{equation*}
Now consider the standard ternary quadratic $\mathcal{O}_{S}$-space
$V = H(\mathcal{O}_{S}) \bot\langle1\rangle$, i.e., with
$B_{q}= \left(
\begin{array}{c@{\quad}c@{\quad}c}
0 & 1 & 0
\\
1 & 0 & 0
\\
0 & 0 & 1
\end{array}
\right) $. For any $[L] \in\text{Pic~}(\mathcal{O}_{S})/2=\text{Pic~}(
\mathcal{O}_{S})$, the quadratic space $V_{L} = H(L) \bot\langle1
\rangle$ belongs to $\text{Cl}_{S}(\underline{\text{\textbf{O}}}_{V})$
and there are four non-equivalent classes in
$\text{Cl}_{S}(\underline{\text{\textbf{O}}}_{V})$. For example,
$A_{\langle x,y \rangle}^{-1} = \left(
\begin{array}{c@{\quad}c}
x & -y/x
\\
y & -x
\end{array}
\right) $ induces the form represented by $\left(
\begin{array}{c@{\quad}c@{\quad}c}
2xy & -2x^{2}-1 & 0
\\
-2x^{2}-1 & 2y & 0
\\
0 & 0 & 1
\end{array}
\right) $, being not $\mathcal{O}_{S}$-isomorphic to $V$, as
$\langle x,y \rangle$ is not principal.
\end{example}


\subsection{$|S|>1$}

If $|S|>1$ then an $\mathcal{O}_{S}$-regular quadratic space $V$ will possess multiple ($2^{|S|-1}$) genera, 
and if $\text{rank}(V) = 3,4$ some of them may be anisotropic, i.e., contain anisotropic representatives only.  

\begin{example}
\label{|S|=2}
Let $C$ be the projective line defined over $\mathbb{F}_{3}$, so
$K=\mathbb{F}_{3}(t)$, and let $S=\{t,t^{-1}\}$, so $\mathcal{O}_{S}=
\mathbb{F}_{3}[t,t^{-1}]$, being the ring of regular functions on the
multiplicative group $\underline{\mathbb{G}}_{m}$, thus it is a PID. The
ternary $\mathcal{O}_{S}$-space $V = \langle1,-1,-t \rangle$ with
$q(x,y,z) = x^{2}-y^{2}-tz^{2}$ is isotropic, e.g., $q(1,1,0)=0$. It is
properly isomorphic over $\mathcal{O}_{S}(i)$ (being a scalar extension
of $\mathcal{O}_{S}$ thus an \'{e}tale one), by $\text{diag}(1,i,-i)$
to the anisotropic form $V'=\langle1,1,t \rangle$. Both are
$\mathcal{O}_{S}$-unimodular as $\det(q)=t \in\mathcal{O}_{S}^{\times}$, 
but belong to two distinct genera (there are exactly
$2^{|S|-1}=2$ such, cf. Proposition~\ref{Wittexactness}). 
As $\mathcal{O}_{S}$ is a PID, 
$\text{Pic}~(\mathcal{O}_{S})=1$, and so
according to Theorem~\ref{isotropicclassification} there is only one class in 
$\text{Cl}_{S}(\underline{\text{\textbf{O}}}_{V})$. 
The hyperbolic plane $\langle 1,-1 \rangle$ has trivial Witt invariant,
and it is orthogonal to $\langle -t \rangle$ in $V$, thus $w(q)=0$ (cf. \cite[IV, Prop.~8.1.1, 1), 3)]{Knu}), so $[q]$ corresponds to the
trivial element in $_{2}{\mathrm{Br}}(\mathcal{O}_{S})$. To compute the
Clifford algebra of $(V',q')$, we choose the natural basis
\begin{equation*}
\{e_{1}=(1,0,0), \ e_{2}=(0,1,0), \ e_{3}=(0,0,1)\}.
\end{equation*}
The embedding $i:V' \hookrightarrow\text{\textbf{C}}(q')$ satisfies the
relations $i(v)^{2}=q'(v) \cdot 1 \ \forall v \in V'$ which imply:
\begin{equation*}
i(v) i(u) + i(u) i(v) = B_{q'}(u,v) \ \ \forall u,v \in V'.
\end{equation*}
Since $\{e_{i}\}_{i=1}^{3}$ are orthogonal, this means that:
\begin{equation*}
i(e_{i})i(e_{j}) = - i(e_{j})i(e_{i})
\ \
\forall1 \leq i \neq j \leq3,
\end{equation*}
so we may choose (as $q'$ is anisotropic the obtained quaternion algebra
is not split):
\begin{equation*}
i(e_{1}) = \left(
\begin{array}{c@{\quad}c}
1 & 0
\\
0 & -1
\end{array}
\right) , \ i(e_{2}) = \left(
\begin{array}{c@{\quad}c}
0 & 1
\\
1 & 0
\end{array}
\right) , \ i(e_{3}) = \left(
\begin{array}{c@{\quad}c}
0 & \sqrt{-t}
\\
-\sqrt{-t} & 0
\end{array}
\right) .
\end{equation*}
Then $\{1,i(e_{1})i(e_{2}),i(e_{2})i(e_{3}),i(e_{1})i(e_{3})\}$ is a
basis of $\text{\textbf{C}}_{0}(q')$ (cf. \cite[V.~\S3]{Knu}):
\begin{equation*}
\text{\textbf{C}}_{0}(q') = \left\langle 1, \ \left(
\begin{array}{c@{\quad}c}
0 & 1
\\
-1 & 0
\end{array}
\right) , \ \left(
\begin{array}{c@{\quad}c}
-\sqrt{-t} & 0
\\
0 & \sqrt{-t}
\end{array}
\right) , \ \left(
\begin{array}{c@{\quad}c}
0 & \sqrt{-t}
\\
 \sqrt{-t} & 0
\end{array}
\right) \right\rangle
\end{equation*}
and $w(q') = [\text{\textbf{C}}_{0}(q')]$ is the non-trivial element in
$_{2}{\mathrm{Br}}(\mathcal{O}_{S})$. 
\end{example}

{\bf Acknowledgements:}
I thank L.~Gerstein, B.~Kahn,
B.~Kunyavski\u{\i} and R.~Parimala for valuable discussions concerning
the topics of the present article. I would like to give special thanks
to P. Gille, my postdoc advisor in Camille Jordan Institute in the
University Lyon 1 for his fruitful advice and support. 

%
%
%
%
%

\begin{thebibliography}{[groups]}


\bibitem[SGA4]{SGA4} M.\ Artin, A. \ Grothendieck, J.-L. Verdier, {\em Th\'eorie des Topos et Cohomologie \'Etale des Sch\'emas} (SGA 4) LNM, Springer, 1972/1973. 

\bibitem[Bas]{Bas} H.\ Bass, {\em Clifford algebras and spinor norms over a commutative ring}, Amer. J. Math., {\bf 96}, 1974, 156--206.

\bibitem[Bau]{Bau} M.\ L.\ Bauer, {\em The arithmetic of certain cubic function fields}, Mathematics of Computation, {\bf 73}, No. 245, 387-–413. 

\bibitem[Bit]{Bit} R.\ A.\ Bitan, {\em The Hasse principle for bilinear symmetric forms over a ring of integers of a global function field}, J. Number Theory, 2016, 346--359. 

\bibitem[BK]{BK}  A.\ J.\ Berrick, M.\ E.\ Keating, {\em An introduction to rings and modules with K-theory in view}, Cambridge University Press, 2000. 

\bibitem[BP]{BP} A.\ Borel, G.\ Prasad, {\em Finiteness theorems for discrete subgroups of bounded covolume in semi-simple groups}, Publ. Math. IHES {\bf 69} 1989, 119--171.

\bibitem[Cla]{Cla} P.\ L.\ Clark, {\em Commutative Algebra}, http://math.uga.edu/\textasciitilde pete/integral2015.pdf

\bibitem[CGP]{CGP} V.\ Chernousov, P.\ Gille, A.\ Pianzola {\em A classification of torsors over Laurent polynomial rings}.  

\bibitem[Con]{Con} B.\ Conrad, {\em Math 252. Properties of orthogonal groups}, http://math.stanford.edu/\textasciitilde conrad/252Page/handouts/O(q).pdf

\bibitem[SGA3]{SGA3} M.\ Demazure, A.\ Grothendieck, {\em S\'eminaire de G\'eom\'etrie Alg\'ebrique du Bois Marie - 1962-64 - Sch\'emas en groupes}, 
Tome II, R\'e\'edition de SGA3, P.\ Gille, P.\ Polo, 2011.  


\bibitem[EKV]{EKV} H.\ Esnault, B.\ Kahn, E.\ Viehweg, {\em Coverings with odd ramification and Stiefel-Whitney classes}, Journal f\"ur die reine und angewandte Mathematik {\bf 441}, 1993, 145--188. 

\bibitem[GP]{GP} P.\ Gille, A.\ Pianzola, {\em Isotriviality and \'etale cohomology of Laurent polynomial rings}, Journal of Pure and Applied Algebra {\bf 212}, 2008, 780--800. 

\bibitem[GS]{GS} P.\ Gille, T.\ Szamueli, {\em Central simple algebra and Galois cohomology}, Cambridge studies in adv. math. {\bf 101}, 2006. 

\bibitem[Ger]{Ger} L.\ J.\ Gerstein, {\em Unimodular quadratic forms over global function fields}, J. Number Theory {\bf 11}, 1979, 529--541.

\bibitem[Gir]{Gir} J.\ Giraud, {\em Cohomologie non ab\'elienne}, Grundlehren math. Wiss., Springer-Verlag Berlin Heidelberg New York, 1971.

\bibitem[Gon]{Gon} C.\ D.\ Gonz\'alez-Avil\'es, {\em Quasi-abelian crossed modules and nonabelian cohomology}, J. Algebra {\bf 369}, 2012, 235--255. 


\bibitem[Gro]{Gro} A.\ Grothendieck, {\em Le groupe de Brauer III: Exemples et compl\'ements}, Dix Expos\'es sur la Cohomologie des Sch\'emas, North-Holland, Amsterdam, 1968, 88-188.  

\bibitem[Hard]{Hard} G.\ Harder, {\em \"Uber die Galoiskohomologie halbeinfacher algebraischer Gruppen, III}, J.\ Reine Angew.\ Math. {\bf 274/275} 1975, 125--138.

\bibitem[Hart]{Hart} R.\ Hartshorne, {\em Algebraic Geometry}, Springer, 1977.  

\bibitem[Knu]{Knu} M.\ A. \ Knus, {\em Quadratic and hermitian forms over rings, Grundlehren}, der mat. Wissenschaften {\bf 294}, 1991, Springer. 


\bibitem[Lam]{Lam} T.\ Y.\ Lam, {\em Introduction to quadratic forms over fields}, Graduate Studies in Mathematics {\bf 67}, 2005. 

\bibitem[Len]{Len} H.\ W.\ Lenstra, {\em Galois theory for schemes}, http://websites.math.leidenuniv.nl/algebra/GSchemes.pdf

\bibitem[Mil]{Mil} J.\ S.\ Milne, {\em \'Etale Cohomology}, Princeton University Press, Princeton, 1980.

\bibitem[Mor]{Mor} M.\ Morishita, {\em On S-class number relations of algebraic tori in Galois extensions of global fields}, Nagoya Math.\ J. {\bf 124}, 1991, 133--144.

\bibitem[Nis]{Nis} Y.\ Nisnevich, {\em \'Etale Cohomology and Arithmetic of Semisimple Groups}, PhD thesis, Harvard University, 1982.

\bibitem[OMe]{OMe} O.\ T.\ O'Meara, {\em Introduction to Quadratic Forms}, Springer-Verlag, New York Berlin, 1971. 


\end{thebibliography}
\end{document}